\documentclass[article,12pt]{amsart}
\usepackage{amscd}
\usepackage{bbm}
\usepackage{dsfont}
\usepackage{enumerate}
\usepackage{latexsym}
\usepackage{srcltx}
\usepackage{amsfonts}
\usepackage{amssymb}
\usepackage{datetime}
\usepackage{setspace}
\usepackage{enumerate}
\usepackage{amsthm}
\usepackage{graphicx}
\usepackage{epstopdf}
\newtheorem{theorem}{Theorem}%[section]
\newtheorem{proposition}[theorem]{Proposition}

\newtheorem{corollary}[theorem]{Corollary}

\theoremstyle{definition}

\newtheorem{definition}[theorem]{Definition}

\newtheorem{remark} [theorem] {Remark}

\newtheorem{problem} [theorem] {Problem}
\input epsf
\begin{document}
\title{On a calculus textbook problem}
\author{Arkady Kitover}
\address{Community College of Philadelphia, 1700 Spring Garden Street, Philadelphia, PA 19130}
\email{akitover@ccp.edu}
\author{Mehmet Orhon}
\address{Department of Mathematics and Statistics, University of New Hampshire, Durham,
NH 03824}
\email{mo@unh.edu}
\subjclass[2010]{Primary 97160; Secondary 97140}

\maketitle

\markboth{A.K.Kitover and M.Orhon}{Calculus textbook problem}
\begin{abstract} We consider generalizations of a well known elementary problem. A wire of the fixed length is cut into two pieces, one piece is bent into a circle and the second one into a square. What dimensions of the circle and the square will minimize their total area?

\end{abstract}

The following problem or its variants appear in many popular textbooks on elementary calculus (see e.g~\cite[Page 338, Problems 37 and 38]{St},  or~\cite[ Page 264, Problem 35]{La}, or~\cite[Page 284, Problem 14]{An} ).

\begin{problem} \label{pr1} Suppose we have a wire of length $L$; we cut the wire into two pieces and bend one piece into a circle of radius $r$ and the second - into a square of side $a$. What values of $r$ and $a$ will minimize the total area $S = \pi r^2 + a^2$?

\end{problem}

 Problem~\ref{pr1} is, of course, elementary and does not even require calculus. The relation $2\pi r +4a = L$ allows us to write $S$ as a quadratic function of $r$ and it is immediate to see that the solution is given by
 $$ a = 2r = \frac{L}{\pi + 4}.$$
 The most interesting thing here is, of course, that the circle can be inscribed into the square.

 The goal of this short paper is to look at the same problem from a slightly more general point of view. We hope that the simple observations below might be of interest to students of calculus as well as to some of our colleagues teaching it.

 \begin{definition} \label{d1} Let $C$ be a simple closed rectifiable curve on the plane $\mathds{R}^2$. Let $L$ be the arc length of $C$ and $A$ the area of the bounded region with the boundary $C$. Let $\lambda$ be a positive number. We will say that $C$ belongs to the class $[\lambda]$ if
 $$ A = \lambda L^2.$$

 \end{definition}

 \begin{remark} \label{r2} It follows from the isoperimetric inequality that $0 < \lambda \leq \frac{1}{4\pi}$ and that $\lambda  = \frac{1}{4\pi}$ if and only if $C$ is a circle. Recall also that if $C$ is a regular polygon with $m$ sides, $m \geq 3$, then  $\lambda = \frac{\cot{\pi/m}}{4m}$.

\begin{remark} \label{r4} The famous proof by Steiner that \textit{if the solution of the isoperimetric problem on the plane exists} then it must be a circle can be found in many books on recreational mathematics (see e.g.~\cite{RT}). A complete (including the proof of existence of the solution), short, and elementary proof can be found in~\cite{De}.

An elementary approach to isoperimetric inequality in three dimensions is contained in the books~\cite{Bl} and~\cite{Kr} (unfortunately not translated into English).

For the modern approach to the isoperimetric inequality (including higher dimensions) the reader is referred to the fundamental monograph~\cite{Fe}.

\end{remark}

 \end{remark}

 \begin{proposition} \label{p1} Let $n$ be a positive integer,  $n \geq 2$, and let $L, \lambda_1, \ldots , \lambda_n$ be positive real numbers. Let $C_1, \ldots , C_n$ be simple closed rectifiable curves with respective arc lengths being $L_i$ such that
 \begin{enumerate}
   \item $C_i \in [\lambda_i]$.
   \item $L_1 + \ldots + L_n =L$.
   \item The bounded regions $G_1, \ldots , G_n$ with respective boundaries $C_1, \ldots , C_n$ are pairwise disjoint.
 \end{enumerate}
 Then the total area $A = \sum \limits_{i=1}^n A_i$ of the region $\bigcup \limits_{i=1}^n G_i$ takes the smallest value if and only if
 $$\lambda_1 L_1 = \lambda_2 L_2 = \ldots = \lambda_n L_n.$$
 In particular, if we denote by $s_i$ the reciprocal of $\lambda_i$, $i= 1, \ldots , n$ then
 $$ L_i =\frac{s_iL}{s_1 + s_2 + \ldots + s_n}, i= 1, \ldots , n,$$
 and therefore the smallest value of $A$ is
 $$A_{min} = \frac{L^2}{s_1+ \ldots + s_n}. $$
 \end{proposition}

 \begin{proof} The proof we present here makes use only of the principle of mathematical induction and the elementary fact that a quadratic function of one variable with a positive leading coefficient takes its smallest value at the vertex.

 (1) The base of induction: $n=2$. In this case $A = \lambda_1 L_1^2 + \lambda_2 L_2^2 = \lambda_1 L_1^2 + \lambda_2(L - L_1)^2 =
 (\lambda_1 + \lambda_2)L_1^2 -2\lambda_2LL_1 + \lambda_2L^2$. This quadratic function of $L_1$ takes its smallest value only at the vertex and therefore $A$ takes the smallest value if and only if $L_1 = \frac{\lambda_2 L}{\lambda_1 + \lambda_2} = \frac{s_1L}{s_1 + s_2}$. Thus the base of induction is proved.

 (2) The induction step: assume that our statement is true for a natural $n \geq 2$; we claim that it remains true for $n+1$. Using the induction assumption we can write that
 $$A = \lambda_1 L_1^2 + \ldots + \lambda_n L_n^2 + \lambda_{n+1} L_{n+1}^2 \geq \lambda_1 \hat{L}_1^2 + \ldots + \lambda_n \hat{L}_n^2 + \lambda_{n+1} L_{n+1}^2, $$
 where
 $$ \hat{L}_i = \frac{s_i(L - L_{n+1})}{s_1 + \ldots + s_n}, \; i=1, \ldots , n,$$
 and the inequality is strict unless $L_i = \hat{L}_i, i= 1, \ldots , n$.

 An elementary computation shows that
 $$ A \geq \frac{(L - L_{n+1})^2}{s_1 + \ldots + s_n} + \frac{L_{n+1}^2}{s_{n+1}}. \eqno{(1)}$$
 The right hand side of $(1)$ is a quadratic function of $L_{n+1}$ :
  $$f(L_{n+1}) = \big{(} \frac{1}{s_1 + \ldots + s_n} + \frac{1}{s_{n+1}}\big{)} L_{n+1}^2 - \frac{2L}{s_1+ \ldots + s_n}L_{n+1} + \frac{L^2}{s_1 + \ldots + s_n}.$$
  This quadratic function takes its smallest value only at the vertex where
  $$L_{n+1} = \frac{s_{n+1}L}{s_1 + \ldots + s_{n+1}}.$$
 \end{proof}

 \begin{remark} \label{r3}

 \begin{enumerate}
   \item We can avoid the use of mathematical induction in the proof of Proposition~\ref{p1}. But then the price we have to pay is to assume that the student is familiar with the fact that a continuous function of $n$ variables attains its smallest value on a compact subset of $\mathds{R}^n$. Notice that in our case $A$ is obviously a continuous function of $L_1, \ldots , L_n$ considered on the compact subset $K$ of $\mathds{R}^n$, where
       $$K = \{(L_1, \ldots , L_n) : L_i \geq 0, i=1, \ldots , n, L_1 + \ldots L_n = L. \} $$
        See the corresponding reasoning in the proof of Proposition~\ref{p2}.

       It might also be of interest to notice that because in Proposition~\ref{p1} $A$ is a \textit{quadratic} function of $n$ variables the fact that it attains its smallest value on $K$ can be proved without involving  either compactness or mathematical induction by means of elementary linear algebra (see~\cite[Theorem 7]{Ja}).
   \item Instead of involving rectifiable curves we can speak in a calculus class of piecewise continuously differentiable curves or in a precalculus class of polygons with sides being either segments of straight lines or circular arcs.
   \item The statement and proof of Proposition~\ref{p1} will remain the same if instead of regions bounded by simple closed curves we consider the following more general situation. Let $O$ be a bounded open subset of $\mathds{R}^2$ such that its boundary is the union of an at most countable set of rectifiable simple curves of finite total length.
 \end{enumerate}

 \end{remark}

 Let us recall that a polygon is called \textit{tangential} if we can inscribe a circle into it, i.e. if there is a circle tangent to every side of the polygon. Of course all triangles and regular polygons are tangential. The interested reader is referred to the Wikipedia articles ``Tangential polygons"  and ``Tangential quadrilaterals" as well as the sources cited there.

 The next corollary of Proposition~\ref{p1} is a direct generalization of Problem~\ref{pr1}.

 \begin{corollary} \label{c1}  Let $n$ be a positive integer,  $n \geq 2$, and let $C_i$, $1 \leq i \leq n$, be either a convex tangential polygon such that the radius of the inscribed circle is $r_i$, or a circle of radius $r_i$. Let the regions $D_i$ bounded by the curves $C_i$ be pairwise disjoint.
  Then the total area $A$ of the set $\bigcup \limits_{i=1}^n D_i$ takes its smallest value if and only if $r_1 = r_2 = \ldots = r_n$.

 \end{corollary}

 \begin{proof}  The statement of Corollary~\ref{c1} follows immediately from Proposition~\ref{p1} and the formula $A_i = \frac{1}{2} r_i L_i$.
 \end{proof}

 We now proceed with generalizations involving surface areas and volumes in $\mathds{R}^3$.

 Let us consider a surface $C \subset \mathds{R}^3$ with the following two properties.

 \noindent (1) $C$ is homeomorphic to the sphere $S^2$.

 \noindent (2) $C$ is piecewise smooth (i.e. piecewise of class $C^{(1)})$.

  Condition (1) together with the Jordan - Brouwer separation theorem (see e.g.~\cite{Li}) guarantees that $\mathds{R}^3 \setminus C$ is the union of two disjoint regions $E$ and $F$, the region $E$ is bounded in $\mathds{R}^3$, and its boundary is $C$. In particular, the volume of $E$ is finite and we will denote it by $V_C$.

  Condition (2) guarantees that the surface area of $C$ is well defined and finite. We will denote it by $A_C$. Moreover, the volume of $C$ in $\mathds{R}^3$ is $0$.

 We will say that a surface $C$ satisfying conditions (1) and (2) belongs to the class $[\lambda]$, where $\lambda$ is a positive real number, if
 $$ V_C = \lambda (A_C)^{3/2}.$$

 Notice that $0 < \lambda \leq \frac{1}{6\sqrt{\pi}}$, where equality is attained if and only if $C$ is a two-dimensional sphere (see Remark~\ref{r4}).

 \begin{proposition} \label{p2} Let $n$ be a positive integer number and $A$ be a positive real number. Let $C_1, \ldots , C_n$ be surfaces in $\mathds{R}^3$ such that

\noindent (a) $C$ is homeomorphic to the sphere $S^2$, $i=1, \cdots , n$.

\noindent (b) $C$ is piecewise smooth (i.e. piecewise of class $C^{(1)})$, $i=1, \cdots , n$.

 \noindent (c) $$C_i \in [\lambda_i], i = 1, \ldots, n.$$

\noindent (d) The regions $E_1, \ldots , E_n$, where $E_i$ is the bounded region with the boundary $C_i$, are pairwise disjoint.

 \noindent (e)
 $$ A_{C_1} + \ldots A_{C_n} = A.$$
 Then the total volume
 $$ V = V_{C_1} + \ldots + V_{C_n}$$
 takes its smallest value if and only if
 $$\lambda_1^2 A_{C_1} = \lambda_2^2 A_{C_2} = \ldots = \lambda_n^2 A_{C_n}$$

 \end{proposition}

 \begin{proof} Consider first the case, when $n=2$. Let us denote $A_{C_1}$ and $A_{C_2}$ by $A_1$ and $A_2$, respectively. Then the total volume $V$ is the following function of $A_1$,
 $$ V = \lambda_1 A_1^{3/2} + \lambda_2 (A - A_1)^{3/2}, \; 0 \leq A_1 \leq A.$$
 Thus
 $$ \frac{dV}{dA_1} = \frac{3}{2}[\lambda_1 \sqrt{A_1} - \lambda_2 \sqrt{A - A_1}], $$
 and at the critical point we have
 $$\lambda_1^2 A_1 = \lambda_2^2 (A- A_1) = \lambda_2^2 A_2.$$
 Moreover,
 $$ \frac{d^2 V}{dA_1^2} = \frac{3}{4} \Big{(} \frac{\lambda_1}{\sqrt{A_1}} + \frac{\lambda_2}{\sqrt{A-A_1}} \Big{)} > 0 \; \mathrm{on} \; [0, A],$$
 and therefore the critical point is the only point in $[0, A]$ where the function $V = V(A_1)$ attains its absolute minimum.

 Now consider the general case. Let us again write $A_i$ instead of $A_{C_i}$, $i =1, \ldots , n$. The function
 $$ V = \sum \limits_{i=1}^n \lambda_i A_i^{3/2}$$
 is continuous on the compact subset $K$ of $\mathds{R}^n$ defined as
 $$K = \{(A_1, \ldots , A_n) : A_i \geq 0 \; \mathrm{and} \; A_1 + \ldots + A_n = A \},$$
 and therefore attains its minimum value on this set. Let $(\tilde{A}_1, \ldots ,\tilde{A}_n)$ be a point where the minimum is attained and assume  contrary to our statement that there are $i$ and $j$, $ 1 \leq i < j \leq n$ such that $\lambda_i^2 \tilde{A}_i \neq \lambda_j^2 \tilde{A}_j$. Let $\hat{A}_i$ and $\hat{A}_j$ be such positive real numbers that $\hat{A}_i + \hat{A}_j = \tilde{A}_i + \tilde{A}_j$ and
 $\lambda_i^2 \hat{A}_i = \lambda_j^2 \hat{A}_j$. Then by the first step of the proof the value of $V$ at the point of $\mathds{R}^n$ where the $i^{th}$ and $j^{th}$ coordinates are changed from $\tilde{A}_i$ and $\tilde{A}_j$ to $\hat{A}_i$ and $\hat{A}_j$, respectively will be strictly less then the value at $(\tilde{A}_1, \ldots , \tilde{A}_n)$, a contradiction.
 \end{proof}

 \begin{corollary} \label{c2} Let $n$ be a positive integer, $n \geq2$. Assume that for every $i \in [1 : n]$ $C_i$ is either a sphere of radius $r_i$ or the surface of a convex tangential polyhedron with the radius of the inscribed sphere being $r_i$. In particular, $C_i$ can be the surface of one of the platonic solids. Assume that the bounded regions $E_i$ with boundaries $C_i$ are pairwise disjoint.  Then the volume $V$ of the region $\bigcup \limits_{i=1}^n E_i$  takes its smallest value if and only if $r_1 = r_2 = \ldots = r_n$.
  \end{corollary}

  \begin{proof} The proof follows from Proposition~\ref{p2} and the formula $V(E_i) = \frac{1}{3} r_i A_i$, where $A_i$ is the surface area of the surface $C_i$.

  \end{proof}

  In addition to the case of tangential polyhedra the following special case might be of interest.

  \begin{proposition} \label{p4} Let $n$ be a positive integer, $n\geq2$, Let $C_i$, $i=1 \ldots , n$ be a simple closed rectifiable curve in $\mathds{R}^2$ with the length $L_i$. Let $C_i \in [\lambda_i]$ (see Definition~\ref{d1}). Assume that the closed bounded regions $D_i$ with the boundaries $C_i$ have pairwise disjoint interiors.
Let $E_i$ be the closed right ``cylindrical'' solid in $\mathds{R}^3$, $E_i = D_i \times [0,h]$ where $h$ - the height of the cylinder - is a positive number not depending on $i$. Assume that the total surface area of the cylinders $E_i, i = 1, \ldots , n$ has a fixed value $S$.

Then the total volume $V$ of all the cylinders will take the smallest value if and only if
 $$\lambda_1 L_1 = \lambda_2 L_2 = \ldots = \lambda_n L_n.$$

  \end{proposition}

\begin{proof} Because $h$ is fixed we need to minimize
 $$V/h = A_1 + \ldots + A_n = \lambda_1 L_1^2 + \ldots + \lambda_n L_n^2$$
 subject to the constraint
  $$(L_1 + \ldots + L_n)h +2(A_1 + \ldots + A_n) = (L_1 + \ldots + L_n)h +2(\lambda_1 L_1^2 + \ldots + \lambda_n L_n^2) = S. $$
 The method of Lagrange multipliers provides
 $$2\lambda_i L_i = \gamma (h + 4\lambda_i L_i), i = 1, \ldots, n, \eqno{(2)}$$
 where $\gamma$ is the Lagrange multiplier. Let $1 \leq i < j \leq n$. Then
 $$ 2(\lambda_j L_j - \lambda_i L_i) = 4\gamma (\lambda_j L_j - \lambda_i L_i).$$
 If $\lambda_j L_j - \lambda_i L_i \neq 0$ then $\gamma = 1/2$ and from $(2)$ we get $h = 0$, a contradiction.

 Strictly speaking we are not done because we need to prove that the conditions $\lambda_i L_i = \lambda_j L_j, 1 \leq i \leq j \leq n$, define a minimum of $V$ and that this minimum is unique. As in the proof of Proposition~\ref{p2} it is enough to prove our statement for the case $n=2$.
 Thus we consider $V/h = \lambda_1 L_1^2 + \lambda_2 L_2^2$ subject to the constraint $(L_1 + L_2)h + 2\lambda_1 L_1^2 + 2\lambda_2 L_2^2 = S$.
 The corresponding border Hessian (see~\cite[Page 383]{Ch}) is the determinant
 \[ \left| \begin{array}{ccc}
0 & 4\lambda_1 L_1 + h & 4\lambda_2 L_2 +h \\
4\lambda_1 L_1 + h & 2\lambda_1 & 0 \\
4\lambda_2 L_2 + h & 0 & 2\lambda_2 \end{array} \right |.\]
It is easy to see that this determinant is equal to $-\lambda_1(4\lambda_2 L_2 +h)^2 - \lambda_2(4\lambda_1 L_2 +h)^2 < 0$. Therefore (see again~\cite[Page 383]{Ch}) the condition $\lambda_1 L_1 = \lambda_2 L_2$ is not only necessary but also sufficient for the corresponding point to be a local minimum of $V$. Finally, notice that $V$ is a strongly convex function (indeed, its Hessian is the diagonal matrix with eigenvalues $2\lambda_1$ and $2\lambda_2$ ) and hence its local minimum is global and unique.
\end{proof}

\begin{corollary} \label{c4}  Let $n$ be a positive integer, $n\geq2$, Let $C_i$, $i=1 \ldots , n$ be either a tangential polygon with the radius of the inscribed circle equal to $r_i$ or a circle of radius $r_i$. Assume that the closed bounded regions $D_i$ with the boundaries $C_i$ have pairwise disjoint interiors.
Let $E_i$ be the closed right ``cylindrical'' solid in $\mathds{R}^3$, $E_i = D_i \times [0,h]$ where $h$ - the height of the cylinder - is a positive number not depending on $i$. Assume that the total surface area of the cylinders $E_i, i = 1, \ldots , n$ has a fixed value $S$.

Then the total volume $V$ of all the cylinders will take the smallest value if and only if
 $$r_1 = r_2 = \cdots = r_n$$

\end{corollary}

  In the case of higher dimensions we can repeat almost verbatim what was said above about the case of $\mathds{R}^3$. Namely, consider $m > 3$. Let $C$ be a hypersurface in $\mathds{R}^m$, i.e. a submanifold of $\mathds{R}^m$ of dimension $m-1$. Assume that

  \noindent (1) $C$ is homeomorphic to the hypersphere $S^{m-1}$.

 \noindent (2) $C$ is piecewise smooth (i.e. piecewise of class $C^{(1)}$).

 Then we can speak about the surface area $A$ of $C$ and the volume $V$ of the bounded region in $\mathds{R}^m$ with the boundary $C$. We will say that $C \in [\lambda]$ if $V = \lambda A^{\frac{m}{m-1}}$.

 Notice that in light of the isoperimeric inequality and the well known formulas for the surface area $A$ of an $(m-1)
 $-dimensional sphere and the volume $V$ of an $m$-dimensional ball in $\mathds{R}^m$ of radius $R$,
 $$ A = \frac{m \pi^{m/2} R^{m-1}}{\Gamma(1 + m/2)}, $$
 $$ V = \frac{ \pi^{m/2} R^m}{\Gamma(1 + m/2)}, $$
 we have

 $$ 0 < \lambda \leq \sqrt[m-1]{\frac{\Gamma(1 + m/2)}{\pi^{m/2} m^m}}, $$
 where the equality is attained if and only if $C$ is a hypersphere.

 \begin{proposition} \label{p3} Let $n, m$ be positive integers, $n, m \geq 2$ and $A$ be a positive real number. Let $C_1, \ldots , C_n$ be hypersurfaces in $\mathds{R}^m$, such that

 \noindent (a)  $C_i$ is homeomorphic to the hypersphere $S^{m-1}$, $i= 1, \cdots , n$.

 \noindent (b)  $C_i$ is piecewise smooth (i.e. piecewise of class $C^{(1)}$), $i = 1, \cdots , n$.

\noindent (c)  $$C_i \in [\lambda_i], i = 1, \ldots, n.$$
(d) The bounded regions $E_1, \ldots , E_n$ with the boundaries $C_i$  are pairwise disjoint.

 \noindent (e)
 $$ A_{C_1} + \ldots A_{C_n} = A.$$
 Then the total volume
 $$ V = V_{C_1} + \ldots + V_{C_n}$$
 takes its smallest value if and only if
 $$\lambda_1^{m-1} A_{C_1} = \lambda_2^{m-1} A_{C_2} = \ldots = \lambda_n^{m-1} A_{C_n}.$$

 \end{proposition}

 \begin{corollary} \label{c3} Let $n, m$ be positive integers, $n, m \geq 2$ and $A$ be a positive real number. Let $C_1, \ldots , C_n$ be hypersurfaces in $\mathds{R}^m$, such that

 \noindent (a)  $C_i$ is either a hypersphere of radius $r_i$ or a convex tangential polytop (in particular, a regular polytop) with the radius of the inscribed hypersphere $r_i$.

\noindent (b) The bounded regions $E_1, \ldots , E_n$ with the boundaries $C_i$  are pairwise disjoint.

 \noindent (e)The total surface area of the hypersurfaces $C_i$ is fixed.
 $$ A_{C_1} + \ldots A_{C_n} = A.$$
 Then the total volume
 $$ V = V_{C_1} + \ldots + V_{C_n}$$
 takes its smallest value if and only if
 $$r_1 = r_2 = \cdots = r_n.$$

 \end{corollary}

 \begin{remark} \label{r1} It is well known that there are six non-isomorphic regular polytopes in $\mathds{R}^4$ and only three in $\mathds{R}^m$ if $m \geq 5$. For details the reader is referred to~\cite{Co}.

 \bigskip

 \textbf{Acknowledgement}. We are very grateful to Eric Grinberg who carefully read a draft of this paper and made numerous suggestions implemented in the current version.

 \end{remark}

\end{document}